\documentclass[12pt]{amsart}
\usepackage{amsmath, amsfonts, latexsym, array, amssymb}
\usepackage{enumitem}
\usepackage{psfrag}
\usepackage[all]{xypic}
\usepackage[all]{xy}
\usepackage{graphicx}
\usepackage[svgnames]{xcolor}
\usepackage{epstopdf}
\usepackage{hyperref}

\usepackage{graphicx}

\DeclareMathOperator{\dist}{dist}

\usepackage{ulem}

\newtheorem{thm}{Theorem}[section]

\newtheorem{que}[thm]{Question}

\newtheorem{defn}[thm]{Definition}
\newtheorem{lem}[thm]{Lemma}

\newtheorem{prop}[thm]{Proposition}

\newtheorem{rem}[thm]{Remark}

\newtheorem{claim}[thm]{Claim}

\newcommand\N{{\mathbb{N}}}
\newcommand\E{{\mathbb{E}}}
\newcommand\R{{\mathbb R}}

\newcommand{\al}{\alpha}
\newcommand{\ep}{\varepsilon}

\newcommand{\ZZ}{\mathbb{Z}}
\newcommand{\Z}{\mathbb{Z}}
\newcommand{\interior}{\mathrm{int}}
\newcommand{\Rep}{\mbox{Rep}}

\title[Hyperbolic sets]{Nonlocally maximal and premaximal hyperbolic sets}

\author{T.~Fisher, T.~Petty, and S.~Tikhomirov}

\address{T.~Fisher, Department of Mathematics, Brigham Young University, Provo, UT 84602}
\email{tfisher@mathematics.byu.edu}

\address{T.~Petty, Department of Mathematics, Brigham Young University, Provo, UT 84602}
\email{taylor.michael.petty@gmail.com}

\address{S.~Tikhomirov, Max Planck Institute for Mathematics in the sciences, Inselstrasse 22, 04103, Leipzig, Germany;
Chebyshev Laboratory, Saint-Petersburg State University, 14th Line 29B, Vasilyevsky Island, St.Petersburg 199178, Russia}
\email{sergey.tikhomirov@gmail.com}

\thanks{The authors would like to thank the Sixth International Conference on Differential and Functional Differential Equations during which much of the paper was prepared. T.F.\ is supported by Simons Foundation grant \# 239708. S. T.\ is partially supported by Chebyshev Laboratory under
Russian Federation Government grant 11.G34.31.0026, JSC ``Gazprom neft'', Saint Petersburg
State University research grant 6.38.223.2014, Russian Foundation of Basic Research
15-01-03797a and German-Russian Interdisciplinary Science Center (G-RISC)}	

\begin{document}

\subjclass[2000]{37D20, 37D05, 37C05}
\keywords{hyperbolic sets, hyperbolic flow, premaximal, locally maximal, isolated}
\commby{}

\begin{abstract} We prove that for any closed manifold of dimension 3 or greater that there is an open set of smooth flows that have a hyperbolic set that is not contained in a locally maximal one.
Additionally, we show that the stabilization of the shadowing closure of a hyperbolic set is an intrinsic property for premaximality. Lastly, we review some results due to Anosov that concern premaximality.

\end{abstract}
\maketitle

\section{Introduction}\label{secIntro}
Since the 1960s the study of hyperbolic sets has been a cornerstone in the field of dynamical systems. These sets are remarkable not only in their complexity, but also in the fact that they persist under perturbations. Additionally, for a point in a hyperbolic set the derivative of the map at this point gives information on the local dynamics for the original nonlinear map.

As a reminder, for a diffeomorphism $f:M\rightarrow M$, a compact invariant set $\Lambda$ is hyperbolic for $f$ if $T_\Lambda M=\mathbb{E}^s\oplus \mathbb{E}^u$ is a $Df$-invariant splitting such that $\mathbb{E}^s$ is uniformly contracted and $\mathbb{E}^u$ is uniformly expanded by $Df$.


Anosov was one of the pioneers in studying hyperbolic sets. Indeed, if the entire manifold is a hyperbolic set for a diffeomorphism, then the diffeomorphism is called {\it Anosov}. This is the best understood class of hyperbolic sets due to its rigid structure.

The next class of hyperbolic sets that are well understood are those that are locally maximal, where $\Lambda$ is {\it locally maximal} for a diffeomorphism if there exists a neighborhood $U$ of $\Lambda$ such that
$$\Lambda=I_f(U)= \cap_{n \in \mathbb{Z}}f^n(U).$$ Hence, such sets are the maximal invariant set within $U$.
A question that was posed in the 1960s, and stated for instance in~\cite{KH}, is the following:

\begin{que}\label{q.main}
If $\Lambda$ is a hyperbolic set and $U$ is a neighborhood of $\Lambda$, then is there a locally maximal hyperbolic set $\tilde{\Lambda}$ such that $\Lambda\subset \tilde{\Lambda}\subset U$?
\end{que}

It was shown by Crovisier~\cite{Cro} that there is a hyperbolic set on the 4-torus that is never included in a locally maximal set. Later, in~\cite{Fis1} it was shown that any compact boundaryless manifold with dimension bigger than 1 has a $C^r$ open set of diffeomorphisms where $1\leq r\leq \infty$ such that each diffeomorphism in the open set contains a hyperbolic set that is not included in a locally maximal one.

Our first result is an extension of the results in~\cite{Fis1} to hyperbolic flows.
As a reminder, for a smooth flow $\phi:M\rightarrow M$, a compact invariant set $\Lambda$ is hyperbolic for $\phi$ if $T_\Lambda M=\mathbb{E}^s\oplus\mathbb{E}^c\oplus \mathbb{E}^u$ is a flow invariant splitting such that $\mathbb{E}^s$ is uniformly contracted by the flow, $\mathbb{E}^c$ is the flow direction, and $\mathbb{E}^u$ is uniformly expanded by the flow. A set $\Lambda$ is locally maximal for the flow $\phi$ if there is an open set $U$ containing $\Lambda$ such that
$$
\Lambda=I_\phi(U)= \cap_{t \in \mathbb{R}}\phi_t(U).
$$

\begin{thm}\label{t.flows}
Let $M$ be a compact, boundaryless $C^r$ manifold for $1\leq r\leq \infty$ with $\dim M \geq 3$ and $\mathcal{X}^k(M)$ be the set of $C^k$ flows on $M$ where $1\leq k\leq r$. Then there exists a $C^k$ open set of flows on $M$ such that each flow contains a hyperbolic set not contained in a locally maximal one.
\end{thm}

\begin{rem}
Note that if $\dim M \leq 2$ then every hyperbolic set for a smooth flow is a finite union of hyperbolic closed trajectories and hence it is locally maximal.
\end{rem}

However, there is a good reason for Question~\ref{q.main}, since all previously known examples of hyperbolic sets were included in locally maximal ones. In this paper we examine conditions under which a hyperbolic set, $\Lambda$, is included in a locally maximal hyperbolic set within an arbitrarily small neighborhood of $\Lambda$.

Following the terminology introduced by Anosov in~\cite{Ano2} we define a
hyperbolic set $\Lambda$ for a diffeomorphism to be {\it premaximal} if for any open set $U$ containing $\Lambda$ there is a locally maximal hyperbolic set $\tilde{\Lambda}$ such that $\Lambda\subset \tilde{\Lambda}\subset U$. In~\cite{Ano2} Anosov proves that any zero-dimensional hyperbolic set for a diffeomorphism is premaximal, and in~\cite{Ano1} Anosov proves there is an intrinsic property for premaximal hyperbolic sets for diffeomorphisms. Moreover the following holds.
\begin{thm} \cite{Ano5}\label{t.anosov}
Let $f:M\rightarrow M$ and $f':M'\rightarrow M'$ be diffeomorphisms, $\Lambda$ a hyperbolic set for $f$, $\Lambda'$ a hyperbolic set for $f'$, and $h:\Lambda\rightarrow \Lambda'$ a homeomorphism such that $h\circ f=f'\circ h$. Then there exists neighborhoods $V \subset U$ of $\Lambda$ and $V' \subset U'$ of $\Lambda'$ and continuous injective equivariant maps $h_1: I_f(U) \to M'$ and $h_2: I_{f'}(U') \to M$ such that $h_1|_{\Lambda} = h$ and
\begin{align*}
 h_1(I_f(V)) & \subset I_{f'}(U'), \quad & h_2(I_{f'}(V')) & \subset I_{f}(U) \\
 h_1 \circ f|_{I_f(V)} &= g \circ h_1|_{I_f(V)}, \quad &  f \circ h_2|_{I_{f'}(V')} &= h_2 \circ f'|_{I_{f'}(V')}\\
 h_2 \circ h_1|_{I_f(V)} & = id, \quad & h_1 \circ h_2|_{I_{f'}(V')} & = id.
\end{align*}
\end{thm}

The above theorem shows that $f|_{\Lambda}$ defines the set of trajectories that lie in a sufficiently small neighborhood of $\Lambda$. However, in~\cite{Ano1} the specific intrinsic property for premaximality is not stated.

We extend result of \cite{Ano1} to the case of flows and prove the premaximality is an intrinsic property for hyperbolic sets for flows.

Let $X$ and $X'$ be vector fields on smooth compact Riemannian manifolds $M$ and $M'$ respectively. Denote by $\varphi$ and $\varphi'$ flows generated by them.

Let $\Lambda$ and $\Lambda'$ be hyperbolic sets for $X$ and $X'$ respectively. We say that $\Lambda$ and $\Lambda'$ are \textit{conjugated} if there exists a homeomorphism $h: \Lambda \to \Lambda'$ and a map $\alpha: M \times \R \to \R$ such that
\begin{equation}\label{eqal}
h \circ \varphi_t(x) = \varphi'(\alpha(x, t), h(x)), \quad x \in \Lambda, t \in \R.
\end{equation}
and
$
\alpha(x, \cdot)
$
a reparameterization for each $x\in\Lambda$.
In that case there exists function $\beta : M' \times \R \to \R$, satisfying the following
$$
\beta(h(x), \alpha(x, t)) = t, \quad x \in \Lambda, t \in \R
$$
$$
\alpha(h^{-1}(x'), \beta(x', t')) = t', \quad x' \in \Lambda', t' \in \R.
$$ 
\begin{thm}\label{thmPremaxFlow}
Let $\Lambda$ and $\Lambda'$ be hyperbolic sets for vector fields $X$ and $X'$ respectively. Assume that $\Lambda$ and $\Lambda'$ are conjugated. Then $\Lambda$ is premaximal if and only if $\Lambda'$ is premaximal.
\end{thm}

Below we provide an equivalent condition to premaximality for diffeomorphisms.
Before stating the result we define some important terms involving shadowing.
Let $a > 0$ be an expansivity constant for some neighborhood of $\Lambda$ and let $\delta_0 > 0$ be such that
any $\delta_0$-pseudotrajectory in $\Lambda$ can be $a/2$-shadowed by an exact trajectory (definitions are given in the next section). Note that due to expansivity the shadowing trajectory is unique.

For $\Lambda$ a hyperbolic set for a diffeomorphism and $\delta \in (0, \delta_0)$ the {\it shadowing closure} (or {\it $\delta$-shadowing closure}) of $\Lambda$ is
$$\mathrm{sh}(\Lambda, \delta)=\overline{\{y\in M: y \textrm{ shadows a }\delta\textrm{-pseudo orbit in }\Lambda\}}.$$
For a fixed $\delta>0$ we can construct a sequence of shadowing closures $\Lambda_0, \Lambda_1, ..., $ where $\Lambda_0=\Lambda$ and $\Lambda_j=\mathrm{sh}(\Lambda_{j-1}, \delta)$ for $j\in\mathbb{N}$. We say a shadowing sequence {\it stabilizes} if $\Lambda_j=\Lambda_{j+1}$ for all $j\geq N$ where $N\in\mathbb{N}$.

\begin{thm}\label{t.shadowing}
Let $f\in\mathrm{Diff}(M)$ and $\Lambda$ be a hyperbolic set for $f$. The following statements are equivalent
\begin{itemize}
\item[(1)] $\Lambda$ is premaximal;
\item[(2)] for any neighborhood $U$ of $\Lambda$ the shadowing closure stabilizes inside $U$ for some $\delta>0$.
\end{itemize}
\end{thm}
Note that due to Theorem \ref{t.anosov} the second property in Theorem \ref{t.shadowing} is intrinsic.

The paper proceeds as follows. In Section~\ref{s.background} we review relevant background on hyperbolicity and flows. In Section~\ref{s.nonlocally} we prove Theorem~\ref{t.flows}. In Section~\ref{s.premaximal} we review the results of Anosov in~\cite{Ano1, Ano2, Ano3, Ano4, Ano5} and prove Theorems, \ref{thmPremaxFlow}, \ref{t.shadowing}.

\section{Background}\label{s.background}

\subsection{Hyperbolic sets for flows}
We first review properties
of hyperbolic sets for flows.


\begin{defn}

Let $X$ be a metric space and $\phi$ a continuous flow on $X$. Then for $x\in X$, we define the \textit{stable set} $$W^{s}(x):=\{y\in X: \lim_{t\to\infty} d(\phi_t(x),\phi_t(y))=0\}.$$ Further, for $\varepsilon>0$ the \textit{$\varepsilon$-stable set} is
$$W^{s}_{\varepsilon}(x):=\{y\in W^{s}(x): d(\phi_t(x),\phi_t(y))\leq\varepsilon \text{ for all } t\geq 0\}.$$ Note that the unstable sets $W^{u}(x)$ and $W^{u}_\varepsilon(x)$ are defined identically under the flow $\phi_{-t}$. Furthermore, we define the \textit{center-stable set} $$ W^{cs}(x):=\phi_t(W^s(x))|_{t\in\mathbb{R}}=\bigcup_{\substack{y\in \phi_t(x) \\ t\in\mathbb{R}}} W^s(y).$$ The center-unstable set of $x$ is defined to be the center-stable set of $x$ under $\phi_{-t}.$ We will also use the notation $W^s_{\text{loc}}$ to mean $W^s_\varepsilon$ for sufficiently small $\varepsilon$ (dependening on the context) and $W^u_{\text{loc}}$ similarly to mean  $W^u_\varepsilon$ for small $\varepsilon$.

\end{defn}

In the case where $X$ is a manifold and $\phi_t$ a $C^r$ flow, the stable set of a point of a hyperbolic set is a $C^r$ submanifold of $X$. Note that the stable manifold is an embedded copy of $\mathbb{R}^k$ where $k=$ dim $\mathbb{E}^s(x)$. The same applies for the unstable sets, center-stable sets, and center-unstable sets. Also note that the stable and unstable manifolds vary continuously both on each other and on the relevant point.

\begin{defn}\label{localproductstructure}
For a metric space $X$ and a flow $\phi,$ a set $\Gamma\subset X$ is said to have a \textit{local product structure} if for all $\varepsilon>0$ there exists a $\delta>0$ such that given $x,y\in\Gamma$ with $d(x,y)<\delta$ we have, for some real $|t|<\varepsilon,$ a unique point $\mathcal{S}(x,y):=b\in W_{\varepsilon}^u(\phi_t(x))\cap W_{\varepsilon}^s(y)$.
\end{defn}



The following lemma is also critical to the paper. Note that this lemma is almost always stated and proved for maps, but is in fact true for flows as well (see \cite{BG} and \cite{Kifer}).

\begin{lem}
A hyperbolic set $\Gamma$ has a local product structure if and only if it is locally maximal.
\end{lem}

In the case that $\Gamma$ is locally maximal and hyperbolic, then $x,y\in\Gamma$ implies $\mathcal{S}(x,y)\subset\Gamma$.

We also need the notion of the shadowing property. For $\ep > 0$ a map $g: \R \to M$ is an $\ep$-pseudotrajectory if the following holds
$$
d(g(t+\tau), \phi_{\tau}(g(t))) < \ep, \quad t \in \R, |\tau| < 1.
$$
We say that $\ep$-pseudotrajectory $g$ is $\delta$-shadowed by a point $x_0$ if there exists an increasing homeomorphisms $\al: \R \to \R$ satisfying
$$
d(g(t), \phi_{\al(t)}(x_0)) < \delta,
$$
$$
\left|\frac{\al(t_1)- \al(t_2)}{t_1-t_2} - 1\right| < \delta, \quad t_1 \ne t_2.
$$
We will use the following standard result from hyperbolic dynamics.

\begin{defn}
We say that vector field $X$ has an expansivity property in a set $W$ there exists constants $a, \tau_0 > 0$ such that if $x_1, x_2 \in I_X(W)$ and there exists a reparametrisation $\alpha$ such that the following inequalities hold
$$
\dist(\varphi(\alpha(t), x_1), \varphi(t, x_2)) < a, \quad t \in \R,
$$
then $x_2 = \varphi(x_1)$, where $\tau \in (-\tau_0, \tau_0)$.
\end{defn}

\begin{thm}
Let $\Lambda$ be a hyperbolic set for a vector field $X$ and a corresponding flow$\phi$.
\begin{itemize}
  \item There exists neighborhood $W$ of $\Lambda$ such that $X$ has an expansivity property on $W$.
  \item Then there exists a neighborhood $U(\Lambda) \subset W$ such that for any $\delta > 0$ there exists $\ep > 0$ such that any $\ep$-pseudotrajectory $g$ can be $\delta$-shadowed by some point $x_0$.
\end{itemize}
\end{thm}

\begin{defn}
For $(\phi, X)$ a dynamical system, a subset $A$ of $X$ is called an \textit{attractor} if it satisfies the following three conditions.

\begin{enumerate}
\item[(i)] $A$ is forward-invariant under $\phi$; i.e., $x\in A$ implies $\phi_t(x)\in A$ for all $t>0$.
\item[(ii)] There exists a neighborhood of $A$, called the \textit{basin of attraction} of $A$ and denoted $\mathcal{B}(A)$, which consists of all points that tend towards $A$ under $\phi_t$ as $t\to\infty$. In other words, $\mathcal{B}(A)=\{x:$ for any open neighborhood $N$ of $A,$ $\exists\ T>0 \ni \phi_t(x)\in N\ \forall t>T\}$.
\item[(iii)] No proper subset of $A$ satisfies conditions (i) and (ii).
\end{enumerate}

\end{defn}

When an attractor $\Lambda$ is (uniformly) hyperbolic, it will exhibit additional properties:

\begin{itemize}\label{uhaproperties}
\item Periodic points are dense in $\Lambda$.
\item For $x\in\Lambda, W^{cu}(x)\subset\Lambda$.
\item For $x$ a periodic point in $\Lambda,$ $\displaystyle\bigcup_{y\in\mathcal{O}(x)\\} W^{cs}(y)$ is dense in $\mathcal{B}(\Lambda).$
\end{itemize}

We will need the following technical result, known in the literature as the Inclination Lemma, or $\lambda$-lemma. The statement can be found in \cite{AP}. Note that the statement for hyperbolic periodic points would be similar.

\begin{lem}[Inclination Lemma]\label{inclinationlemma}
Let $p\in M$ be a hyperbolic fixed point for a $C^r$ flow $\phi,$ for $r\geq 1,$ with local stable and unstable manifolds $W^s_{\text{loc}}(p)$ and $W^u_{\text{loc}}(p),$ respectively. Fix an embedded disk $B$ in $W^u_{\text{loc}}(p)$ which is a neighborhood of $p$ in $W^u_{\text{loc}}(p),$ and fix a neighborhood $V$ of this disk in $M.$ Let $D$ be a transverse disk to $W^s_{\text{loc}}(p)$ at a point $z$ such that $D$ and $B$ have the same dimension. Write $D_t$ for the connected component of $\phi_t(D)\cap V$ which contains $\phi_t(z),$ for $t\geq 0.$
Then, given $\varepsilon>0$ there exists $T>0$ such that for all $t>T$ the disk $D_t$ is $\varepsilon$-close to $B$ in the $C^r$-topology.
\end{lem}

\subsection{Hyperbolic sets of diffeomorphisms}

The proof of Theorem~\ref{t.shadowing} relies on  shadowing.
Before proceeding to the proof we define these terms and other related terms.

For $f:X\rightarrow X$ a homeomorphism of a metric space and $\varepsilon>0$ an {\it $\varepsilon$-pseudo orbit } is a sequence $\{x_j\}_{l}^m$ where
\begin{itemize}
\item $l\in -\{\infty\}\cup \mathbb{Z}$, $m\in\mathbb{Z}\cup\{\infty\}$, $l<m$, and
\item $d(f(x_{j}), x_{j+1})<\varepsilon$ for all $j\in[m,l]$.
 \end{itemize}
 For an $\varepsilon$-pseudo orbit $\{x_j\}_l^m$ we say this sequence is $\delta$-shadowed by a point $x\in X$ if $d(f^j(x), x_j)<\delta$ for $j \in [m, l]$. The next result is a standard result in hyperbolic dynamics concerning shadowing of pseudo orbits.


\begin{thm}(Shadowing Lemma)
Let $\Lambda$ be a hyperbolic set for $f:M\rightarrow M$ a diffeomorphism. Then there exists neighborhood $U(\Lambda)$ such that for all $\delta>0$ there exists an $\varepsilon>0$ such that if $\{x_j\}_{-\infty}^{\infty} \subset U$ is an $\varepsilon$-pseudo orbit, then there exists $x\in M$ that $\delta$-shadows $\{x_j\}_{-\infty}^{\infty}$.
\end{thm}

Let $f:M\rightarrow M$ be a diffeomorphism and $\Lambda$ be a hyperbolic set for $f$. For $\varepsilon>0$ sufficiently small and $x\in \Lambda$ the
\textit{local stable and unstable manifolds} are respectively:
$$
\begin{array}{llll}
W_{\varepsilon}^{s}(x,f)=\{ y\in M\, |\textrm{ for all }
n\in\mathbb{N}, d(f^{n}(x), f^{n}(y))\leq\varepsilon\},\textrm{
and}\\
W_{\varepsilon}^{u}(x,f)=\{ y\in M\, |\textrm{ for all }
n\in\mathbb{N}, d(f^{-n}(x), f^{-n}(y))\leq\varepsilon\}.
\end{array}$$
The \textit{stable and unstable manifolds} are respectively:
$$
\begin{array}{llll}
W^s(x,f)=\bigcup_{n\geq 0}f^{-n}\left(
W_{\varepsilon}^s(f^n(x),f)\right), \textrm{ and}\\
W^u(x,f)=\bigcup_{n\geq
0}f^{n}\left(W_{\varepsilon}^u(f^{-n}(x),f)\right). \end{array}
$$
For a $C^r$ diffeomorphism $f$ the stable and unstable manifolds of a hyperbolic set are $C^r$
injectively immersed submanifolds.

For $\Lambda$ a hyperbolic set we know that if $\varepsilon$ is sufficiently small and $x,y\in\Lambda$, then $W^s_{\varepsilon}(x)\cap W^u_{\varepsilon}(y)$ consists of at most one point. For such an $\varepsilon>0$ define
$$D_{\varepsilon}=\{(x,y)\in \Lambda\times \Lambda\, |\, W^s_{\varepsilon}(x)\cap W^u_{\varepsilon}(y)\in\Lambda\}$$ and $[\cdot,
\cdot]:D_{\varepsilon}\rightarrow \Lambda$ so that
$[x,y]=W^s_{\varepsilon}(x)\cap W^u_{\varepsilon}(y)$.

We will also need openness of hyperbolicity.

\begin{lem}\label{openhyp}
Let $\Lambda\subset M$ be a hyperbolic set of the diffeomorphism $f: U\to M$. Then for any open neighborhood $V\subset U$ of $\Lambda$ and every $\delta>0$ there exists $\varepsilon>0$ such that if $f': U\to M$ and $d_{C^1}(f|_V, f')<\varepsilon$ there is a hyperbolic set $\Lambda'=f'(\Lambda')\subset V$ for $f'$ and a homeomorphism $h:\Lambda'\to\Lambda$ with $d_{C^0}(\text{Id},h)+d_{C^0}(\text{Id},h^{-1})<\delta$ such that $h\circ f'|_{\Lambda'}=f|_{\Lambda}\circ h$. Moreover, $h$ is unique when $\delta$ is sufficiently small.
\end{lem}



\subsection{Normal hyperbolicity}

For embedding constructions into higher dimensions, we will need the notion of normal hyperbolicity. A normally hyperbolic invariant manifold (NHIM) is a generalization of a hyperbolic fixed point and a hyperbolic set.
Fenichel proved
that NHIMs and their stable and unstable manifolds are persistent under perturbation \cite{fenichel2}, \cite{fenichel3}. We define NHIMs for maps, but the definition for flows is similar (and more technical).

\begin{defn}\label{normalhyp}
Let $M$ be a compact smooth manifold and $f: M\to M$ a diffeomorphism. Then an $f$-invariant submanifold $\Lambda$ of $M$ is said to be a \textit{normally hyperbolic invariant manifold} if there exist constants $0<\mu^{-1}<\lambda<1$ and $c>0$ such that
\begin{itemize}
\item $T_{\Lambda} M=T\Lambda\oplus \E^s\oplus\E^u$
\item $(Df)_x\E^s(x)=\E^s(f(x))$ and $(Df)_x\E^u(x)=\E^u(f(x))$ for all $x\in\Lambda$
\item $\|Df^n v\|\leq c\lambda^n\|v\|$ for all $v\in\E^s$ and $n>0,$
\item $\|Df^{-n} v\|\leq c\lambda^n\|v\|$ for all $v\in\E^u$ and $n>0,$ and
\item $\|Df^n v\|\leq c\mu^{|n|}\|v\|$ for all $v\in T\Lambda$ and $n\in\Z.$
\end{itemize}
\end{defn}

Adapting the above for flows gives us an important result (\cite[p.~215]{fenichel}) which says that if a $C^r$ vector field $Y$ in some $C^1$ neighborhood of our original vector field $X$ (equated with a flow $\phi,$ under which $M$ is invariant) there is a $C^r$ manifold $M_Y$ invariant under $Y$ and $C^r$ diffeomorphic to $M.$ An immediate consequence of this is that the dynamics on $M_Y$ under the vector field $Y$ are a perturbation of the dynamics of $M$ under $X.$




\section{Nonlocally maximal sets for flows}\label{s.nonlocally}

The foundation of the proof of Theorem~\ref{t.flows} is the classic Plykin map, see for instance~\cite[p.~537-41]{KH} for a construction of the Plykin attractor. The first author used this map with some modifications to prove Theorem 1.3 in~\cite{Fis1} on the existence of hyperbolic sets not included in locally maximal ones. We then elaborate on Hunt's PhD dissertation, in which he extends the Plykin map to a flow \cite{hunt}. Using that flow, we apply the construction in~\cite{Fis1} to Hunt's flow and address the technicalities. Figure \ref{f.plykin} shows one iteration of a fundamental domain of the Plykin attractor.
\begin{figure}
\begin{center}
\includegraphics[width=0.9\textwidth]{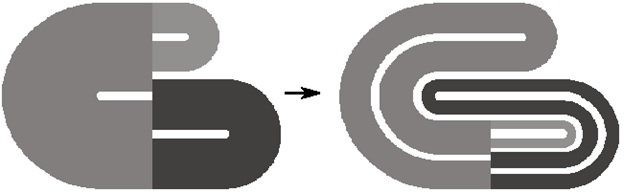}
\caption{The Plykin map.}\label{f.plykin}
\end{center}
\end{figure}

\subsection{Hunt's Flow}


Hunt starts the Plykin attractor construction over from the beginning in the flow case.
The construction is outlined below, but the overarching idea is to view each iterate of the map as a cross-section of a solid 2-torus and connect each point $x$ to its image $f(x)$ via a continuous path around the torus. 


We now include a brief summary of the construction of the flow \cite[p. 53-67]{hunt}.
Hunt breaks the construction into stages. The first is the stretching stage.  The next stage is the folding stage.  The explicit $2\pi$-periodic formula for the flow performs these stages in turn, with smooth transitions between them.

The first stage involves squashing and stretching, which is what gives the flow hyperbolicity.
At this point it is necessary to shift the picture to ensure the original set $A$ flows back into itself. By forcing the flow to satisfy $\phi_1(A)\subsetneq A,$ there are natural constraints placed on the center coordinates.

The next stage is to fold the stretched-out region back into the desired location. Again, this places constraints on the coordinates of the different regions.

After ten constraints have been placed on the region's center, horizontal and vertical shifts, there are still two free variables. 
These are equivalent to the fact that the total horizontal and vertical lengths of the attractor can be decided arbitrarily.


\subsection{Proof of Theorem~\ref{t.flows}}

First we would like to sketch the main steps of the construction of the example. We embed Hunt's Plykin attractor in an invariant solid 2-torus $T_1$. After this we embed $T_1$ into a bigger solid torus $T_2$ (we need this step to be able later to embed the construction into arbitrarily manifold). Next we modify the flow using the construction from \cite{Fis1}. Finally we embed $T_2$ into an arbitrarily manifold of dimension at least 3.

\textbf{Step 1.} As a first step we take Hunt's Plykin flow on the solid 2-torus.
Hunt's flow $\phi$ maps a closed set $A$ (containing the Plykin attractor $\Lambda_a$) strictly into itself -- specifically, $\phi_t(A)\subset\interior(A)$ for all $t>0$ \cite[p.~64]{hunt}. Embed this set $A$ into a solid closed 2-torus, $T_1,$ so $A\subset\interior(T_1).$ Extend the flow $\phi$ to contract in $\interior(T_1\setminus A)$ such that $\phi_t(x)\to A$ as $t\to\infty$ for all $x\in\interior(T_1\setminus A),$ and so that every $x\in\partial T_1$ is period 1.

\textbf{Step 2.} We now have a solid closed 2-torus $T_1$ where $\Lambda_a\subsetneq A\subsetneq T_1,$ where every point in $\interior (T_1)$ asymptotically approaches $\Lambda_a$ in forward time, and every point in $\partial T_1$ is period 1. To later embed this system into a larger manifold while maintaining smoothness, we must enclose $T_1$ in a strictly larger solid closed 2-torus, $T_2.$ We need to extend $\phi$ so that every point on $\partial T_2$ is fixed and the flow is still $C^2.$ For arbitrary $r\in\N,$ this extension is possible in a $C^r$ fashion. First set the boundary conditions: $\phi_t(x)=x$ for all $x$ in a neighborhood of $\partial T_2, t\in\R,$ and $\phi_{k}(x)=x$ for all $x\in\partial T_1, k\in\N.$ In $T_2\setminus T_1,$ in between the boundaries, smoothly vary from fixed points on $T_2$ to period-1 points on $T_1.$ This maintains the smoothness we need while gaining the properties we want.

\textbf{Step 3.} Fix some $p\in\partial T_1.$ Take an open neighborhood $U$ of $\mathcal{O}(p),$ small enough to be disjoint from $\partial T_2$ and the attractor, and alter $\phi$ in the following way to make $p$ a hyperbolic periodic saddle point: alter $\phi$ on $U\cap\partial T_1$ such that $x\in U\cap\partial T_1$ implies $\phi_t(x)\to\mathcal{O}(p)$ as $t\to\infty$ (in the sense that for any $\varepsilon>0$ there exists some $\tau$ such that $t\geq\tau$ implies $d(\phi_t(x),y)<\varepsilon$ for some $y\in\mathcal{O}(p)$. Note that this works because $p$ is periodic). Thus, $W^s(p)=U\cap\partial T_1.$ Again in a smooth way, change $\phi$ in $U\setminus \partial T_1$ to give $p$ an unstable manifold, as in Figure \ref{crosssection}. Note at this point it is clear to see that $p$
is a hyperbolic point of period 1. The center manifold of $p$ is isometric to $S^1$ and each unstable manifold along $\mathcal{O}(p)$ is diffeomorphic to a line segment. Thus we have that $W^{cu}_{\text{loc}}(p)$ is diffeomorphic to a cylinder. Consider $W^{cu}(p)\setminus\mathcal{O}(p),$ which is now (locally) a disjoint union of two cylinders. One of these two components is entirely inside the basin of attraction, and one is entirely outside (reference Figure \ref{crosssection}). Label as $W^*(p)$ the component of $W^{cu}(p)\setminus\mathcal{O}(p)$ which lies entirely inside the basin of attraction (equivalently, $W^*(p)=W^{cu}(p)\cap\interior(T_1)).$ Label as $W^*_0(p)$ the component of $W^u(p)$ which lies entirely inside the basin of attraction. Note that for all $x\in W^*(p)$ we have $\mathcal{O}(x)\subset\interior(T_1).$

\begin{figure}
\begin{center}
\includegraphics[width=5in]{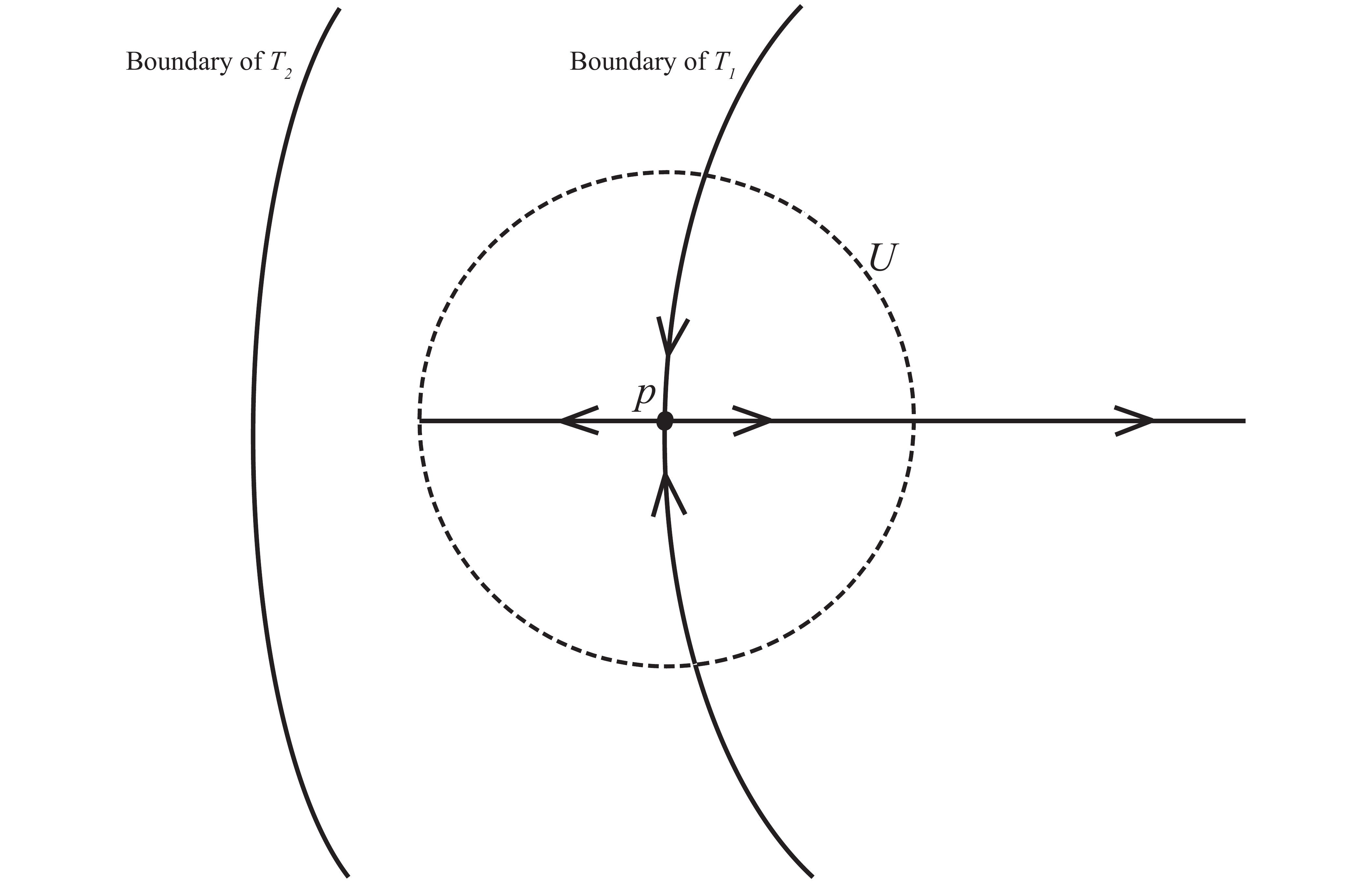}
\caption{View of a cross section at $p.$}\label{crosssection}
\end{center}
\end{figure}

We know that $\overline{W^{cs}(x)}=\overline{W^s(\Lambda_a)}$ for all $x\in\Lambda_a$ (Section \ref{uhaproperties}). By above we have that $W^*(p)\subset\interior(T_1),$ and that $\interior(T_1)=W^s(\Lambda_a).$ Consider some periodic $q_0\in\Lambda_a.$ Then $W^*_0(p)\subset W^*(p)\subset W^s(\Lambda_a) \subset \overline{W^{cs}(q_0)}.$ Given any point of $W^*_0(p),$ then, there must exist some point in $W^{cs}(q_0)$ arbitrarily close to it. Consider some $z\in W^*_0(p).$ Since $q_0$ is period-$R$ for some $R\in\N$ there exists an $R\in\R$ such that for all $\varepsilon>0$ there exists a $t\in [0,R]$ such that $d(\phi_t(q_0),z)<\varepsilon.$ If $z\in W^{cs}(q_0)$ then we're done since then $z\in W^s(\phi_t(q_0))$ for some $t\in [0,R],$ so we can assume $z\notin W^{cs}(q_0).$ Then $z$ must be a limit point of $W^{cs}(q_0)$ 
Perturb the flow in a neighborhood of $z$ (analogously to what is done in Figure \ref{quadratictan}) so that $z\in W^{cs}(q_0).$

Given $z\in W^{u}(p)\cap W^{s}(q)$ for some period-1 $q\in\Lambda_a,$ the manifolds $W^{u}(p)$ and $W^{s}(q)$ can have a transverse intersection, after another perturbation, and we justify this as follows: by construction, we already know that $z\in W^{u}(p)\cap W^{s}(q).$ In \cite[p.~1508]{Fis1}, a perturbation of the map is made in a small neighborhood of $f^{-1}(z),$ which, since the map is continuously differentiable, ensures that $z\in W^u(p)\pitchfork W^s(q)$ in the two-dimensional case. Now in the three-dimensional case, extend this perturbation to the solid 2-torus by perturbing $\phi$ in a sufficiently small neighborhood of $\phi_{-1}(z)$ (analogously to what is done in Figure \ref{quadratictan}) so that $W^{u}(p)\cap W^s(q)$ transversally at some time $t.$

Here we will need two definitions (see \cite[p.~1495]{Fis1}). A hyperbolic set $\Lambda$ for a $C^1$ flow has a \textit{heteroclinic tangency} if there exist $x,y\in\Lambda$ such that $W^s(x)\cap W^u(y)$ contains a point of tangency. A point of \textit{quadratic tangency} for a $C^2$ flow is defined as a point of heteroclinic tangency where the curvature of the stable and unstable manifolds differs at the point of tangency.

By transversality and continuity, as well as the fact that Hunt's flow can be adapted to be $C^2$
\cite[p.~67,70]{hunt}, there must exist a neighborhood $J_0\subset W^{u}_{\text{loc}}(q)$ of $q$ and a neighborhood $I_0\subset W^{u}(p)$ of $z$ such that for each $x\in I_0$ we have $x\in W^s_{\text{loc}}(y)$ for some $y\in J_0.$ Now at some $z'\in I_0\setminus\{z\},$ deform the flow in a sufficiently small neighborhood of $\phi_{-1}(z')$ to create a point of quadratic tangency, $w\in I_0,$ between $W^u(p)$ and $W^{cs}(k)$ for some $k\in J_0$ (see Figure \ref{quadratictan}). Let $I$ be the segment of $W^u(p)$ from $z$ to $w,$ and let $J$ be the segment of $W^u(q)$ from $q$ to $k.$

\begin{figure}
\begin{center}
\includegraphics[width=0.9\textwidth]{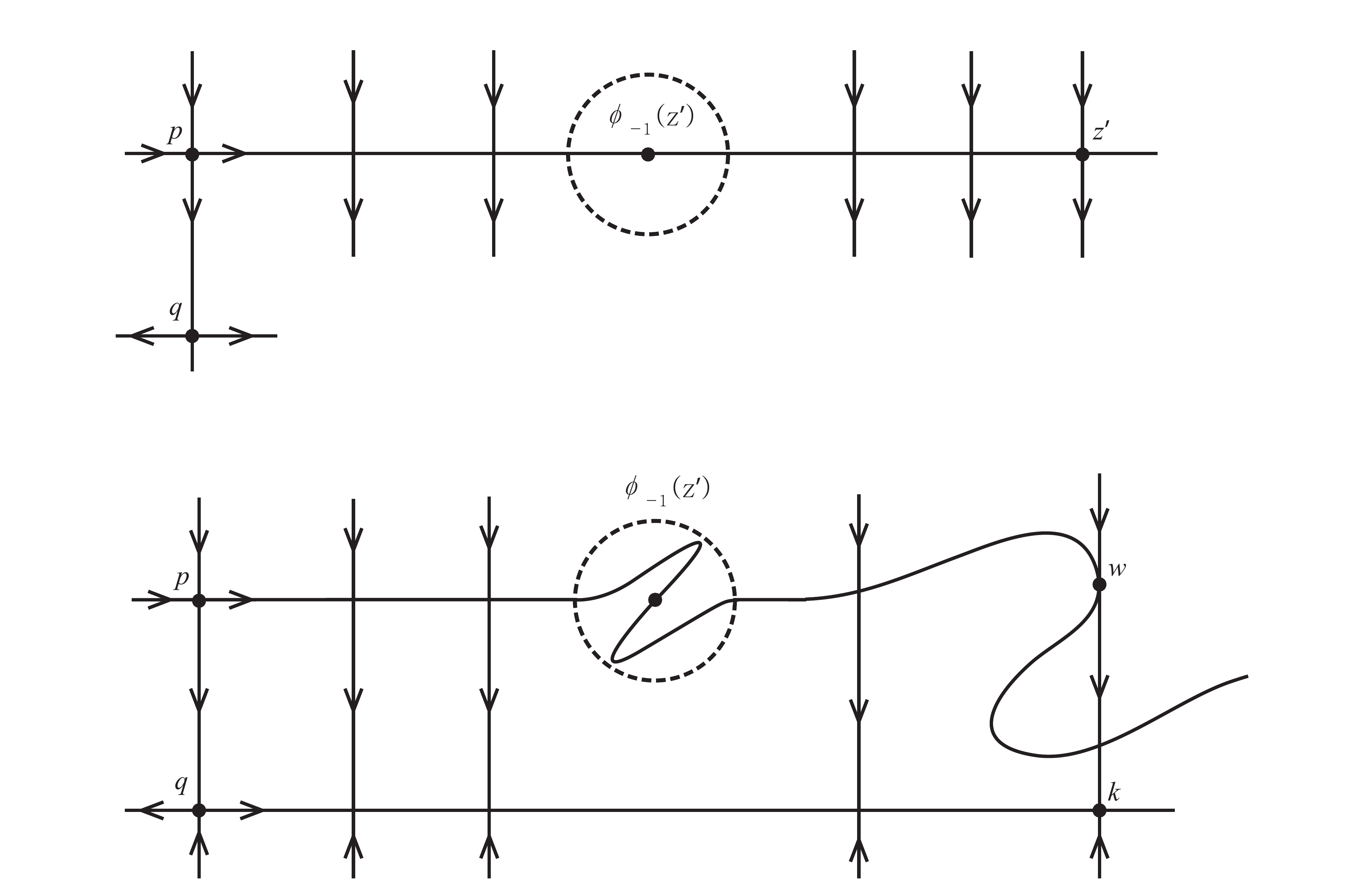}
\caption{Deforming to obtain $w$.}
\label{quadratictan}
\end{center}
\end{figure}

\textbf{Step 4.} We have a flow on a solid torus that is the identity on $\partial T_2.$ To embed this system into any smooth manifold of dimension 3, first embed $T_2$ into a solid sphere $S_3,$ and set $\phi_t(x)=x$ for all $x\in S_3\setminus T_2, t\in\R.$ Scale $S_3$ to be as small as necessary, set $\phi_t(x)=x$ for all $x\in M\setminus S_3, t\in\R,$ and now the example extends to any smooth 3-manifold.

Using normal hyperbolicity (see Definition \ref{normalhyp} and the remarks immediately following), we can embed our example into any smooth manifold $M$ of dimension greater than 3, as follows: first take the solid sphere $S_3$ from above, such that $\phi_t$ is the identity on $\partial S_3.$ For a manifold $M$ of dimension $n>3,$ embed $S_3$ into a solid $n$-sphere $S_n$. In $\interior(S_n\setminus S_3),$ extend $\phi$ so that it contracts sufficiently strongly in all directions towards $S_3$ to make $S_3$ a normally hyperbolic invariant manifold with respect to the flow $\phi.$ We need $\phi$ to fix every point in $\partial S_n$ to ensure the system is easily embeddable into larger manifolds, but we simultaneously need $\phi$ to contract sufficiently strongly to make $S_3$ a normally hyperbolic invariant manifold. Consider a small neighborhood $N$ of $\partial S_n.$ In $N\cap S_n,$ use a smooth bump function to let $\phi$ smoothly vary from sufficiently strong contraction (towards $S_3,$ as previously mentioned) in $S_n\setminus N$ to fixing every point $x\in\partial S_n$ Lastly, define $\phi$ to fix every point in $M\setminus S_n.$ Our example is now embeddable into $M.$

Now that we have combined Hunt's and Fisher's constructions, we need to prove hyperbolicity of the relevant set.

We will show $\Lambda = \Lambda_a\cup {\mathcal{O}(p)}\cup {\mathcal{O}(z)}$ is hyperbolic under $\phi$. Certainly $\mathcal{O}(p)$, $\mathcal{O}(z)$, and $\Lambda_a$ are (at least forward-) invariant under $\phi$, by definition. By construction, $\phi_t(z)$ converges to $\mathcal{O}(q)$ as $t\rightarrow \infty$ and converges to $\mathcal{O}(p)$ as $t\rightarrow -\infty$. Since these are both in $\Lambda$, and we already know $\Lambda_a$ is closed, we have that $\Lambda$ is closed.

Let $\lambda_p, \lambda_q\in (0,1)$ be constants that guarantee hyperbolic behavior in $\mathcal{O}(p)$ and $\Lambda_a$, respectively. Let $y$ be an arbitrary element of $\mathcal{O}(p)$. We assume an adapted metric, so for $t>0$ and $v\in\mathbb{E}^s(y)$, $\|D\phi_t(y) v\| \leq \lambda_p^t\|v\|$, and for $t>0$ and $r\in\mathbb{E}^u(y)$, $\|D\phi_{-t}(y)r\|\leq \lambda_p^t\|r\|$, and similarly for $\mathcal{O}(q)$. Let $\lambda_{max}=\text{ max}(\lambda_p,\lambda_q)$. Then certainly the constant $\lambda_{max}$ guarantees hyperbolicity over $\mathcal{O}(p)$ and $\Lambda_a$, using the above definitions. Now select any $\lambda\in (\lambda_{max}, 1)$. Fix $t>0$. By continuity and the Inclination Lemma (Lemma \ref{inclinationlemma}) there exists an $\varepsilon > 0$ such that for all $x$ where d$(\phi_t(x),\mathcal{O}(p))< \varepsilon$, we have $v\in\mathbb{E}^s(x)$ implying $\|D\phi_t(x) v\| \leq \lambda^t\|v\|$ (similarly for $\E^u(x)$ as well as identical cases for $\mathcal{O}(q)$). In other words, $\phi$ is hyperbolic with constant $\lambda$ for any point $\varepsilon$-close to either $\mathcal{O}(p)$ or $\mathcal{O}(q).$

Since $\displaystyle \lim_{t\to\infty} \phi_t(z) = \mathcal{O}(q)$ and $\displaystyle\lim_{t\to\infty} \phi_{-t}(z) = \mathcal{O}(p),$ there exists a $T>0$ such that $\phi_T(z)$ is $\varepsilon$-close to $\mathcal{O}(q)$ and $\phi^{-T}(z)$ is $\varepsilon$-close to $\mathcal{O}(p)$. By the previous paragraph, we can guarantee hyperbolicity in $\varepsilon$-neighborhoods of $\mathcal{O}(p)$ and $\mathcal{O}(q)$. Now we only need guarantee hyperbolicity outside of those neighborhoods. Acting over the domains $t\in [-T,T]$ and $v\in\{a\in\mathbb{R}^3:\|a\| = 1\}$, the function $\displaystyle \|D\phi_t(z)v\|$ is bounded because it is a continuous function on a (union of) compact domain(s).

Therefore there exists a $C\geq 1$ such that, for any $x\in\Lambda$ and $t>0$, we have $\|D\phi_t(x)v\|\leq C\lambda^t\|v\|$ for all $v\in\mathbb{E}^s(x)$, and $\|D\phi_{-t}(x)w\|\leq C\lambda^t\|w\|$ for all $w\in\mathbb{E}^u(x)$. By definition, $\Lambda$ is hyperbolic.

Now that we have shown the system satisfies the relevant properties. As previously mentioned, our flow $\phi_t$ and Fisher's map $f$ coincide at integer values of time, due to the constant roof function. In other words, $\phi_n(x)=f^n(x)$ for all $n\in\Z, x\in M$ where $M$ is the manifold. Figure \ref{intervals} is helpful to keep in mind throughout the argument.

\begin{figure}
\begin{center}
\includegraphics[width=0.9\textwidth]{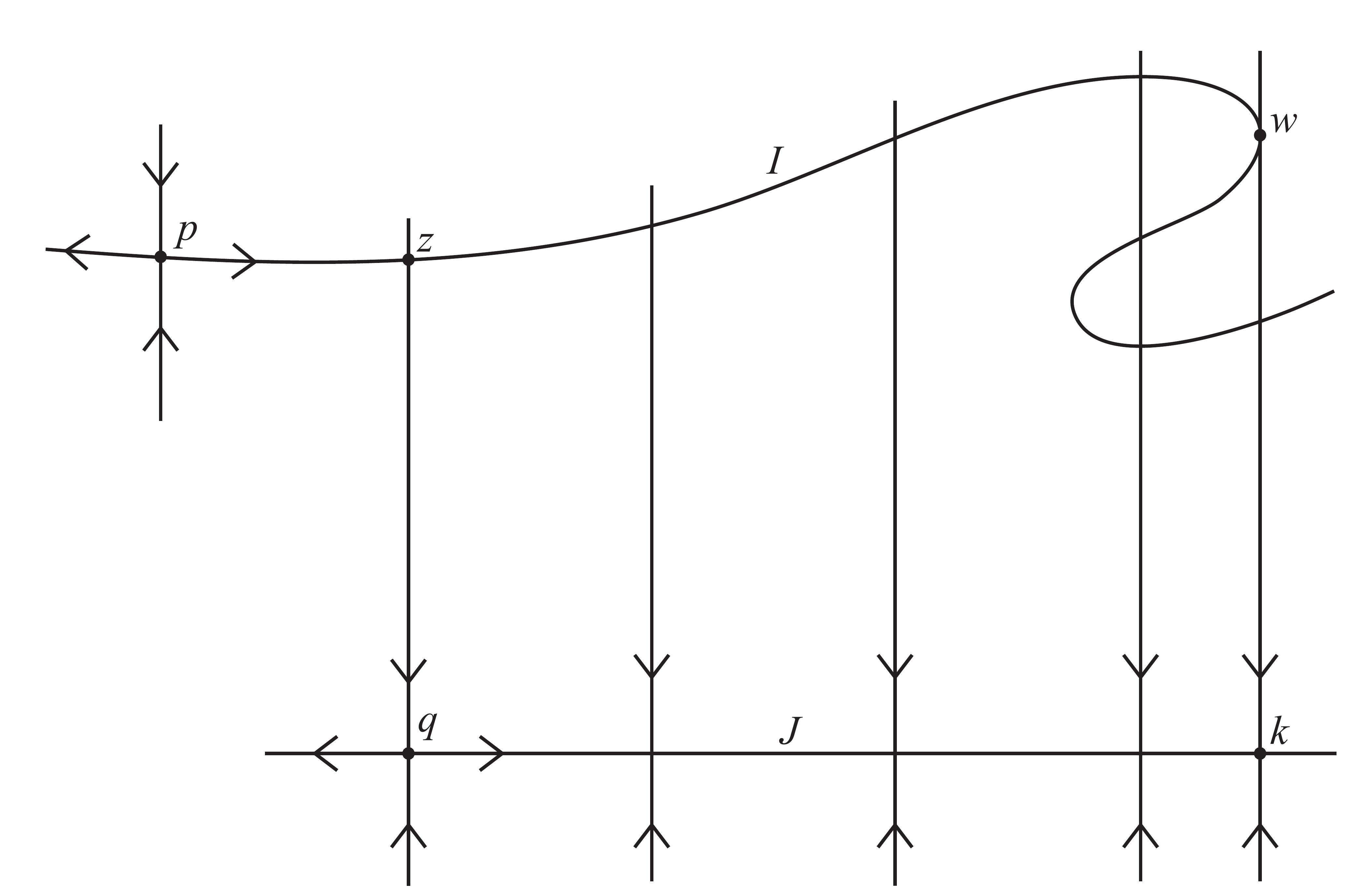}
\caption{Intervals I and J.}
\label{intervals}
\end{center}
\end{figure}

We have a hyperbolic set $\Lambda$ as a subset of a three-dimensional manifold. Since the original diffeomorphism acts on the unit disk, this three-dimensional flow lies in the solid 2-torus. Now suppose $\Lambda\subset\Lambda'$, where $\Lambda'$ is a locally maximal hyperbolic set. It is sufficient to show that some point in $\mathcal{O}(w)\subset\Lambda'$ fails to exhibit continuous splitting of the tangent space, since the quadratic tangency persists in time. Now fix $\delta$ and $\varepsilon$ to satisfy the local product structure.

Define $I$ to be the closed interval from $z$ to $w$ along $W^u(p).$ By construction, every point in $I$ is in the stable manifold of exactly one point in $W^{u}(q)$. Pick $r\in\Z$ such that $t\geq r$ implies $d(a,b)<\delta/2$ for all $a\in\phi_t(I), b\in W^{cu}(q).$ Let $J\subset W^{u}(q)$ be the set of all points $x\in W^{u}(q)$ such that $x\in W^s(\gamma)$ for some $\gamma\in \phi_r(I)$. The stable manifolds connecting $I$ to $J$ depend continuously on points in $I$, by hyperbolicity. Since $\phi$ is the continuous suspension of a diffeomorphism there exists a homeomorphism $\beta: [0,1]\rightarrow \phi_r(I)$. By above, there exists a homeomorphism $\sigma: \phi_r(I)\rightarrow J$. Therefore $\sigma\circ\beta: [0,1]\rightarrow J$ is a homeomorphism. In other words, since $\phi_r(I)$ is a path-connected interval, $J$ is also a path-connected interval.

Certainly $\phi_r(z)\in \phi_r(I)\cap\Lambda'$ since $z\in I$ and $\mathcal{O}(z)\subset \Lambda\subset\Lambda'$. Consider some $\alpha\in J$ such that $d(\alpha,q)<\delta/2$ (guaranteed to exist since $J$ is path-connected). Since $d(\phi_r(z),q)<\delta/2$, then $d(\phi_r(z),\alpha)<\delta$ by the triangle inequality. By the local product structure, there exists exactly one point $\mathcal{S}(\phi_r(z),\alpha)\in W_{\varepsilon}^u(\phi_r(z))\cap W^s_{\varepsilon}(\alpha)\subset \Lambda'.$ Since we are adapting Fisher's flow to the solid 2-torus $T_1$ with a constant roof function -- $\phi_n(x)=f^n(x)$ for all $n\in\Z, x\in T_1$.  Note,  there need be no $\varepsilon$ time shift as in the definition of local product structure. In this case, then, our definition for local product structure for flows can be taken to be that for all $\varepsilon>0$ there is a $\delta>0$ such that given $x,y\in\Gamma$ with $d(x,y)<\delta$ then $\mathcal{S}(x,y)=W^u(x)\cap W^s(y)=\{b\}\subset \Gamma.$ (Compare with Definition \ref{localproductstructure}.) Note that in the perturbed case, the $\varepsilon$ time shift in the flow definition of local product structure is indeed necessary but the argument still works since the unstable and stable manifolds depend continuously on each other, by hyperbolicity.

Take any point $e$ along the interval lying in $J$ from $q$ to $\alpha$. Thus $d(e,q)<d(\alpha,q)<\delta/2$ so similarly we get a unique $\mathcal{S}(\phi_r(z),e)\in \phi_r(I)\cap\Lambda'$. Since $e$ was arbitrary, we now have a closed interval of points from $\phi_r(z)$ to $\mathcal{S}(\phi_r(z),\alpha)$ along $\phi_r(I)\cap\Lambda'$.

Thus $z':=\mathcal{S}(\phi_r(z),\alpha)\in \phi_r(I)\cap\Lambda'$. Inductively repeat the above process along $J$ to see that $\phi_r(I)\subset\Lambda'$. Thus we see that the endpoint $\phi_r(w)$ of the interval $\phi_r(I)$ is also a point in $\mathcal{O}(w)$ where the hyperbolic splitting of the tangent space fails to continuously extend. We thus have found that $\Lambda$ cannot be contained in a locally maximal hyperbolic set, which proves the theorem.

We will see that every aspect of the system is robust under perturbation, so the entire system is as well. For the perturbed system we will use $\tilde{p}$ to denote the continuation of $p,$ and we will similarly denote the continuations of the other aspects of the construction. We will first prove openness for 3-manifolds, and then for the $n$-dimensional case. Since transversality is trivially open, and hyperbolicity is open by Lemma \ref{openhyp}, it is sufficient to show that there remains a point $\tilde{w}\in W^u(\tilde{p})\cap W^s(\tilde{u})$ for some $\tilde{u}\in W^{cu}_{\text{loc}}(\tilde{q}).$ Under a small perturbation Figure \ref{quadratictan} remains identical in effect. By construction, the stable manifolds for all the $x\in W^{cu}_{\text{loc}}(\tilde{q})$ locally foliate the region, so there must exist a point $\tilde{u}\in W^{cu}_{\text{loc}}(\tilde{q})$ and a point $\tilde{w}\in W^{cs}(\tilde{u})\cap W^u(\tilde{p})$ such that the one-dimensional path $W^u(\tilde{p})$ remains tangent to the two-dimensional plane $W^{cs}(\tilde{u})$ at $\tilde{w}$ -- specifically, $T_{\tilde{w}} W^u(\tilde{p})\subsetneq T_{\tilde{w}} W^{cs}(\tilde{u}).$

Since every other part of the system is known to be open, and we have just shown that the curve through $w$ must remain tangent to some center-stable manifold even after perturbation, we have that the entire flow is open in the function space for 3-manifolds. Once there is a perturbation made in the time direction, the roof function is no longer constant so using the local product structure to show $\tilde{w}\in\widetilde{\Lambda}'$ requires use of the time-shift. The time-shift must be bounded by the $\varepsilon$ from the local product structure, but for a sufficiently small perturbation we have the desired bound. The argument then works similarly: pick $\tilde{\alpha}\in W^u(\tilde{q})$ to be $\delta$-close to $\phi_{r}(\tilde{z})$ for $r$ sufficiently large. For some $t$ where $|t|<\varepsilon$ we will get $\mathcal{S}(\phi_{r}(\tilde{z}),\tilde{\alpha})\in {\Lambda'}.$ Continue along $\phi_{r}(\tilde{I})$ as before, using possibly different time-shifts at every iteration, to see $\tilde{w}\in\widetilde{\Lambda}'.$

For the $n$-dimensional case the only change made in the argument is with regards to the dynamics in $S_n,$ the solid $n$-sphere from the construction in which the invariant set $S_3\supset T_2$ was embedded. Using \cite[p.~205]{fenichel}, we make the contraction in $S_n\setminus S_3$ sufficiently strong so that
for a $C^1$ perturbation made to our flow $\phi_t,$ there remains some invariant manifold $\tilde{S_3}$ diffeomorphic to $S_3.$ This means that there is a normally hyperbolic invariant manifold diffeomorphic to the previous one, so the flow $\tilde{\phi}$ restricted to $\tilde{S_3}$ is a small perturbation of $\phi$ restricted to $S_3$. $\Box$

\section{Premaximality}\label{s.premaximal}

Before we proceed to the proof of Theorems~\ref{thmPremaxFlow} and \ref{t.shadowing} let us review results by Anosov on premaximality \cite{Ano1, Ano2, Ano3, Ano4, Ano5}.
As was mentioned in Section \ref{secIntro}, Theorem \ref{t.anosov} implies that premaximality is an intrinsic property.

Let us recall the main idea of the proof of Theorem \ref{t.anosov}. For $\delta > 0$ denote $P(\delta)$ as the set of $\delta$-pseudotrajectories $\{y_n \in \Lambda\}_{n \in \ZZ}$ endowed with the Tikhonov product topology. Consider the shift map $\sigma: P(\delta) \to P(\delta)$ defined as $\sigma(\{y_n\}) = \{y_{n+1}\}$. The Shadowing Lemma implies that for any $\ep > 0$ there exists $\delta > 0$ such that for any $\{y_n\} \in P(\delta)$ there exists point $x$ such that
\begin{equation}\label{eqS1}
  \dist(y_n, f^n(x)) < \ep.
\end{equation}
Note that due to the expansivity property we know for small enough $\ep > 0$ that  such a point $x$ is unique. Fix such an $\ep > 0$ and a corresponding $\delta$ from the shadowing lemma.
Consider the map $T: P(\delta) \to M$ defined by the condition that for $\{y_n\} \in P(\delta)$ the point $x = T(\{y_n\})$ is the unique point satisfying \eqref{eqS1}. It is easy to show that
\begin{equation}\label{eqS2}
  T \circ \sigma = f \circ T.
\end{equation}
Now let us consider $d > 0$ and a small neighborhood $U \subset B(d, \Lambda)$ of $\Lambda$ and a point $z$ such that $O(z) \subset U$. There exists a sequence of points $\{y_n\}$ satisfying
\begin{equation}\label{eqd}
\dist(y_n, f^n(z)) < d.
\end{equation}
For any $\delta > 0$ there exists a $d> 0$ such that inequality \eqref{eqd} implies $\{y_n\} \in P(\delta)$ and 
\begin{equation}\label{eqS3}
T(\{y_n\}) = z.
\end{equation}

We can similarly define for $\Lambda'$ and $f'$ sets $P'(\delta)$ and maps $T'$, $\sigma'$. Note that similarly to \eqref{eqS2} the equality
\begin{equation}\label{eqS4}
T' \circ \sigma' = f' \circ T'
\end{equation}
holds.

For any $\delta'_1 > 0$ there exists $\delta_1 > 0$ such that
\begin{equation}\label{eqSS1}
h(P(\delta_1)) \subset P'(\delta'_1).
\end{equation}
(Recall that $h : \Lambda \to \Lambda'$ is a conjugacy between $\Lambda$ and $\Lambda'$.)
Similarly for any $\delta_2$ there exists $\delta'_2$ such that
\begin{equation}\label{eqSS2}
h^{-1}(P'(\delta'_2)) \subset P(\delta_2).
\end{equation}

Equations \eqref{eqS2}, \eqref{eqS3}, \eqref{eqS4}, \eqref{eqSS1},   and \eqref{eqSS2}
allow us to conclude Theorem \ref{t.anosov}.
For a detailed exposition of the proof we refer the reader to the original paper \cite{Ano5}.


To study premaximality of the zero-dimensional hyperbolic sets we need the following notion.
Let $A$ be an ``alphabet'': $A = \{1, \dots, n\}$ and $\Omega = A^{\mathbb{Z}}$ equipped with the Tikhonov topology and metric
$$
\dist(a, b) = \sum_{i \in \mathbb{Z}} \frac{1}{2^{|i|}}I(a_i, b_i),
$$
where $a = (a_i), b = (b_i) \in \Omega$ and $I(a_i, b_i)$ equal to 0 if $a_i = b_i$ and 1 otherwise. Consider the shift map $\sigma = \Omega \to \Omega$ defined as $(\sigma(a))_i = a_{i+1}$.
Consider some set $W$ of \textit{admissible} words of lengths $k \geq 1$ in the alphabet $A$ and consider $M_W \subset \Omega$ such that all subwords of all $a \in M_{W}$ of length $k$ are admissible. The following was proved in 1960s (see for instance \cite{Alexeev}).
\begin{thm}
 Consider a set $\Lambda \subset \Omega$. The following holds
 \begin{enumerate}
 \item the set $\Lambda$ is locally maximal for $\sigma$ if and only if there exists $k \geq 1$ and set of admissible words $W$ such that $\Lambda = M_W$;
 \item the set $\Lambda$ is premaximal.
 \end{enumerate}
\end{thm}
In \cite{Ano2} it was proved that zero-dimensional hyperbolic sets are topologically conjugated to Bernoulli shifts, which implies the next result.
\begin{thm}
Let $\Lambda$ be a zero-dimensional hyperbolic set of a diffeomorphism $f$.   Then $\Lambda$ is premaximal.
\end{thm}

Burns and Gelfert were able to extend the above result to prove that a 1-dimensional hyperbolic set for a flow is premaximal~\cite[Proposition 8]{BG}.  The proof follows an argument provided by Anosov after personal communications.

In \cite{Ano3} Anosov obtain the following sufficient condition for a hyperbolic set to not be premaximal.
\begin{thm}
 Let $\Lambda$ be a hyperbolic set of $f \in C^1$. Assume that there exists a family of exact trajectories $\xi: \mathbb{Z} \times [0, a] \to M$ such that
 \begin{enumerate}
 \item $\xi_{n+1, t} = f(\xi(n, t))$,
 \item $\xi(0, 0) \in \Lambda$,
 \item $\dist(\xi(n, t), \Lambda) \to 0$ uniformly in $t$ as $|n| \to \infty$, and
 \item there exists $t_1 \in [0, a]$ such that $\xi(0, t_1) \notin \Lambda$.
 \end{enumerate}
 Then $\Lambda$ is not premaximal.
\end{thm}
Note that examples of Crovisier and Fisher satisfy these conditions.

First we prove an analog of Theorem \ref{t.anosov} for flows.

\begin{thm}\label{thmConj}
Let $\Lambda$ and $\Lambda'$ be hyperbolic sets for vector fields $X$ and $X'$ respectively. Assume that $\Lambda$ and $\Lambda'$ are conjugated. Then there exists neighborhoods $U \subset V$ of $\Lambda$ and $U' \subset V'$ of $\Lambda'$,  numbers $\tau_0, \tau'_0 > 0$, and (not necessarily continuous) maps
\begin{align*}
h_1&: I_X(V) \to M', & \quad h'_1&: I_{X'}(V') \to M
\end{align*}
such that 
\begin{equation}\label{eqthm1}
  h_1|_{\Lambda} = h, \quad h'_1|_{\Lambda'} = h^{-1}.
\end{equation}
Then for any $\ep > 0$ there exists $\delta > 0$ such that if $x_1, x_2 \in I_X(U)$ and $\dist(x_1, x_2) < \delta$, then there exists $|\tau'| < \tau_0$ such that
\begin{equation}\label{eqthm21}
      \dist(h_1(x_1), \phi'_{\tau'}(h_1(x_2)))) < \ep
\end{equation}
      and for any $x'_1, x'_2 \in I_{X'}(U')$ and $\dist(x'_1, x'_2) < \delta$, then there exists $|\tau| < \tau_0$ such that
\begin{equation}\label{eqthm22}
      \dist(h'_1(x_1), \phi_{\tau}(h'_1(x_2)))) < \ep,
\end{equation}
  \begin{equation}\label{eqthm31}
  h_1(O(x)) \subset O'(h_1(x)),  \quad x \in I_X(U),
  \end{equation}
  \begin{equation}\label{eqthm32}
  h'_1(O'(x')) \subset O(h'_1(x')), \quad x' \in I_{X'}(U'),
  \end{equation}
\begin{equation}\label{eqthm41}
  h_1 I_X(U) \subset I_{X'}(V'), \; h'_1 I_{X'}(U') \subset I_X(V), \textrm{ and} 
\end{equation}
\begin{equation}\label{eqthm42}
  I_{X'}(U') \subset h_1 I_X(V), \; I_X(U) \subset h'_1 I_{X'}(V').
\end{equation}
For any $x \in I_X(U)$ and  $x' \in I_{X'}(U')$  there exists $|\tau| < \tau_0$ and $|\tau'|<\tau'_0$ such that
\begin{equation}\label{eqthm5}
      h'_1 \circ h_1(x) = \phi(\tau, x), \quad h_1 \circ h'_1(x') = \phi'(\tau', x').
\end{equation}
\end{thm}

\begin{rem}
  Note that \eqref{eqthm21}, \eqref{eqthm22} are analogs of continuity of maps $h_1, h'_1$. Relations \eqref{eqthm5} state that $h'_1$ is almost an inverse of $h_1$. Also, we tried to construct continuous invertible maps with similar properties, but failed due to the nonuniqueness of shadowing.
\end{rem}

\begin{proof}
For a point $x \in M$ and $\ep > 0$ such that $X(x) \ne 0$ denote
$$
L(x, \ep) := \cup_{v \in T_x M, \; |v| < \ep \; v \perp X(x)} \exp_x(v).
$$
Note that for any closed invariant nonsingular set $\Lambda$ there exists $\ep, \tau_1 > 0$ and a neighborhood $U$ such that
$
U \subset \cup_{x \in \Lambda} L(x, \ep),
$
 and 
$
L(x, \ep) \cap L(\varphi(t, x), \ep) = \emptyset$ for $ x \in M, t \in (-\tau_1, \tau_1).
$

Let $W, W'$ be the neighborhoods and $\tau_0, \tau'_0 > 0$ be constants from the expansivity property for $\Lambda$ and $\Lambda'$.

For $\delta > 0$ denote by $P(\delta)$ the set of $\delta$-pseudotrajectories $g(t) \in \Lambda$. For any $\tau \in \R$ consider the mapping $\sigma_{\tau} : P(\delta) \to P(\delta)$ defined by the relation
$$
(\sigma_{\tau}g)(t) = g(t+\tau).
$$

Shadowing and expansivity imply that for any $\varepsilon > 0$ there exists $\delta > 0$ such that for any $g \in P(\delta)$ there exists a unique point $x \in L(g(0), \ep)$ and a (not necessarily unique) reparametrisation $\gamma \in \Rep$ such that
\begin{equation}\label{eq4.5}
\dist(g(t), \varphi(\gamma(t), x)) < \ep.
\end{equation}

Choose such a small $\delta$ (for arbitrarily $\ep$) and define a map $T: P(\delta) \to M$ such that $x = T(g)$ is the unique point satisfying \eqref{eq4.5}.

For $\delta_1 > 0$ we say that pseudotrajectories $g_1, g_2 \in P(\delta)$ are $\delta_1$-rep-close if there exists a reparametrisation $\gamma \in \Rep$ such that
$$
\dist(g_1(t), g_2(\gamma(t))) < \delta_1.
$$

We will use the following properties of the map $T$, which are consequences of shadowing and expansivity properties.
\begin{prop}\label{propT1}
For any $g \in P(\delta)$ and $t \in \R$ there exists $t' \in \R$ such that
\begin{equation}\label{eq5.1}
T \circ \sigma_t(g) = \varphi_{t'} \circ T(g).
\end{equation}
\end{prop}
\begin{prop}\label{propT2}
For small enough $\delta_1, \delta > 0$ if $g_1, g_2 \in P(\delta)$ are $\delta_1$-rep-close then there exists $|\tau| < \tau_0$ such that
    $$
    T g_2 = \phi(\tau, T g_1).
    $$
\end{prop}

For some $\delta_2 > 0$ consider $V = B(\delta_2, \Lambda) \subset W$. For any point $z \in I_X(V)$ and $t \in \R$ let us choose a point $g(t) \in \Lambda$ such that the inclusion $\varphi_t(z) \in L(\delta_2, g(t))$ holds. Note that for small enough $\delta_2$ the map $g(t)$ is a $\delta$-pseudotrajectory. Define map $S: U \to P(\delta)$ as $S(z) := g$. We would like to emphasize that the choice of $g(t)$ is not unique, however for any such choice $T(g) = z$.

We will use the following properties of the map $S$, which are consequences of shadowing and expansivity properties.
\begin{prop}\label{propS1}
\begin{equation}\label{eqS1}
T \circ S = id.
\end{equation}
\end{prop}
\begin{prop}\label{propS2}
For any $\ep > 0$ there exists $T > 0$ and $\delta_1$ such that if $g_1, g_2 \in P(\delta)$ and
$$
\dist(g_1(t), g_2(t)) < \delta_1, \quad |t| < T
$$
then there exists $|\tau| < \tau_0$
$$
\dist(S(g_1), \phi(\tau, S(g_2))) < \ep.
$$
\end{prop}
\begin{prop}\label{propS3}
For any $\delta_1 > 0$ there exists $\delta > 0$ such that for any $g \in P(\delta)$ pseudotrajectories $g$ and $S \circ T g$ are $\delta_1$-rep-close.
\end{prop}

\begin{prop}\label{propS4}
For any $\delta_4 > 0$ there exists neighborhood $U$ of $\Lambda$ such that for any $x \in I_X(W)$ the inclusion $S(x) \in P(\delta_4)$ holds.
\end{prop}


Define similarly $\delta', \delta'_1 > 0$ and maps $T': P(\delta') \to B(\ep', \Lambda')$, $S': B(\delta', \Lambda') \to P(\delta')$.

Now we construct the maps $h_1, h'_1$.
First we will choose neighborhoods $V, V'$ and construct maps $h_1, h'_1$. After this we will choose neighborhoods $U, U'$. Note that we can always decrease $U, U'$.

Let us choose $\delta_2> 0$ such that for any $g \in P(\delta_2)$ the inclusion $h(g) \in P(\delta')$ holds. Now according to Proposition \ref{propS4} let us choose $V$ such that for any $x \in I_X(V)$ the inclusion $S(x) \in P(\delta_2)$ holds. Define the map $h_1: I_X(V) \to M'$ as $h_1:= T' \circ h \circ S$.
Similarly, define $V'$ and map $h'_1: I_{X'}(V') \to M$ as $h'_1 = T \circ h^{-1} \circ S'$.

Among the properties of maps $h_1$, $h'_1$ the most difficult is relation \eqref{eqthm5}. We will give its proof in full details.

Relations \eqref{eqthm1}--\eqref{eqthm42} can be easily deduced (decreasing $U$ and $U'$ if necessarily) from Propositions \ref{propT1}--\ref{propS4}, expansivity of the vector fields in $V, V'$ and continuity of $h$, $h^{-1}$.




We prove only the second equality \eqref{eqthm5}. Note that 
$$h'_1 \circ h_1 = T \circ h^{-1} \circ S' \circ T' \circ h \circ S.$$

Let us choose $\delta'_3 > 0$ such that if $g'_1$ and $g'_2$ are $\delta'_3$-rep-close then $h^{-1} g'_1$, $h^{-1} g'_2$ are $\delta_1$-rep-close.

By Proposition \ref{propS3} let us choose $\delta'_4 > 0$ such that for $g'_1 \in P(\delta'_4)$ pseudotrajectories $g'_1$ and $S' \circ T' g'_1$ are $\delta'_3$-rep-close.

By Proposition \ref{propS4} and the continuity of $h$ we can choose  a neighborhood $U \subset V$ such that for any $x_1 \in I_X(U)$ the inclusion $h \circ S x_1 \in P(\delta'_4)$ holds.

Now we are ready to prove the second equality \eqref{eqthm5}. Let $x_1 \in U $. Set
$$
g_1 := S x_1, \; g'_1 := hg_1,\; x' := T'g'_1, \; g'_2 := S'x', \; g_2 := h^{-1}g'_2, \; x_2 := Tg_2.
$$
By construction of $U$ the inclusion $$
g'_1 \in P(\delta'_4)
$$
holds. Note that
$$
g'_2 = S' \circ T' g'_1,
$$
hence, $g'_1$ and $g'_2$ are $\delta'_3$-rep-close. Note that
$$
g_1 = h^{-1} g'_1, \textrm{ and } g_2 = h^{-1} g'_2.
$$
Hence, $g_1$ and $g_2$ are $\delta_1$-rep-close. Finally, we have
$$
x_1 = T g_1, \quad x_2 = T g_2.
$$
Proposition \ref{propT2} implies the second equality \eqref{eqthm5}.
\end{proof}

\begin{proof}[Proof of Theorem~\ref{thmPremaxFlow}]
We prove that if $\Lambda$ is premaximal then $\Lambda'$ is premaximal. The converse statement can be proven similarly.

Consider neighborhoods $U \subset V$ of $\Lambda$ and $U' \subset V'$ of $\Lambda'$ and maps $h_1, h'_1$ from Theorem \ref{thmConj}.

Assume that $\Sigma \subset U$ is a locally maximal set with isolating neighborhood $W = B(\eta, \Sigma)$, such that $\Sigma' = O(h_1(\Sigma)) \subset U'$. Below we will prove that $\Sigma'$ is locally maximal.

Consider $\eta_1 > 0$ such that for $x \in B(\eta_1, \Sigma)$ and $|\tau| < \tau_0$ the inclusion $\phi_{\tau}(x)$ holds.

Using equality \eqref{eqthm32} and inequality \eqref{eqthm22} for $\ep = \eta_1$ we find $\eta'_2 > 0$ such that if $x' \in W' := B(\eta'_3, \Sigma')$ then $h'_1(x') \in B(\eta_1, \Sigma)$.

Let us prove that $\Sigma' = I_{X'}W'$. Assume contrarily, then there exists $x' \in W' \setminus \Sigma'$ such that $O(x') \subset W'$.
Relations \eqref{eqthm31}, \eqref{eqthm32}, \eqref{eqthm5} imply that
\begin{equation}\label{eqtauincl}
O(h(x')) = \bigcup_{|\tau|< \tau_0, \; x'_1 \in O(x')} \phi_{\tau}(h'_1(x'_1)).
\end{equation}
Note that due to choice of $\eta'_2, \eta_1$ the inclusion $\phi_{\tau}(h'_1(x'_1)) \in W$ holds. Hence $O(h(x')) \subset W$. Local maximality of $\Sigma$ implies that $h(x') \in \Sigma$ and $x' \in \Sigma'$, which leads to a contradiction.

Now we are ready to prove premaximality of $\Lambda'$. Since $\Lambda$ is premaximal there exists sequence $\Lambda_n \subset B(\ep_n, \Lambda) \subset U$ of locally maximal sets, where $\ep_n \to 0$. Arguing similarly to \eqref{eqtauincl} it is easy to show that sets $\Lambda'_n = O(h_1(\Lambda_n))$ satisfies the inclusion $\Lambda'_n \subset B(\ep'_n, \Lambda')$ where $\ep'_n \to 0$ and for large enough $n$ sets $\Lambda'_n$ are locally maximal. Hence $\Lambda'$ is premaximal.
\end{proof}

\subsection{Proof of Theorem~\ref{t.shadowing}}

We now proceed to the proof that a hyperbolic set is premaximal if and only if the shadowing closure stabilizes.
The next statement is not hard to prove, but is essential for our arguments and shows that if a sequence of shadowing closures stabilizes, then it does so at a locally maximal set.

As was claimed in \cite{Ano3} $\Lambda$ is locally maximal if and only if $\Lambda$ has internal shadowing (for sufficiently small $\delta > 0$ any $\delta$-pseudotrajectory consisting of points from $\Lambda$ can be shadowed by a trajectory from $\Lambda$ for small enough $\delta$). It easily leads us to the following.
\begin{prop}\label{p.stabilize}
If $\mathrm{sh}(\Lambda, \delta)=\Lambda$ for some $\delta>0$. Then $\Lambda$ is locally maximal.
\end{prop}

As a reminder if $\Lambda$ is a locally maximal set, then for $\varepsilon>0$ sufficiently small we know that $B_{\varepsilon}(\Lambda)$ is an isolating set for $\Lambda$.

In~\cite{Cro} Crovisier provides an example of a hyperbolic set that is never included in a locally maximal one, and this example shows that to establish premaximality it is not enough to say that a hyperbolic set is included in a locally maximal one or that the shadowing closure stabilizes for some $\delta>0$.

To be more precise, let Let $A $ and $B$ be $2\times 2$ hyperbolic toral automorphisms such that
$A$ has two fixed points $p$ and $q$ and the hyperbolicity in $A$ dominates the hyperbolicity in $B$. Let $F$ be a diffeomorphism on the 4-torus defined by $F_0(x,y)=(Ax, By)$ and fix $r$ a fixed point of $B$. Crovisier proves that if we let $V$ be a sufficiently small neighborhood of $(q,r)$, then the set
$$\Lambda=\bigcap_{n\in\mathbb{Z}} f^n(\mathbb{T}^4-V)$$
is a hyperbolic set and the only locally maximal hyperbolic set that $\Lambda$ can be included in is the entire manifold $\mathbb{T}^4$; so this example is not premaximal. However, $\Lambda$ is included in a locally maximal hyperbolic set and for $\delta>0$ sufficiently large may be in a shadowing closure that stabilizes.

\begin{proof}[Proof of Theorem~\ref{t.shadowing}]
Suppose that $\Lambda$ is a hyperbolic set for a diffeomorphism $f:M\rightarrow M$ and suppose that for any neighborhood $U$ of $\Lambda$ there exists a $\delta>0$ such that the shadowing closure of $\Lambda$ stabilizes inside $U$. Now the previous proposition shows that the stabilizer is a locally maximal set and $\Lambda$ is premaximal.

To prove the other implication let $f:M\rightarrow M$ be a diffeomorphism and
let $\Lambda$ be a premaximal set for $f$ and $V$ be a neighborhood of $\Lambda$. Then there exists some $\eta>0$ such that
$$
B_{\eta}(\Lambda)=\bigcup_{x\in\Lambda} B_{\eta}(x)\subset V
$$
and $\Lambda_{\eta} = I_f(B_{\eta}(\Lambda))$ is hyperbolic.


Let $\tilde{\Lambda}$ be a locally maximal hyperbolic set such that
$$
\Lambda\subset \tilde{\Lambda}\subset \Lambda_{\eta}\subset B_{\eta}(\Lambda)\subset V.
$$ Fix $\varepsilon\in(0,\eta)$ such that $V_{\varepsilon}(\tilde{\Lambda})$ is an isolating neighborhood of $\tilde{\Lambda}$ and fix $\delta\in (0,\varepsilon/2)$ such that every $\delta$-pseudo orbit in $\tilde{\Lambda}$ is $\varepsilon$ shadowed in $\tilde{\Lambda}$ by a unique point in $\tilde{\Lambda}$.

From the choice of constants above we know that $\Lambda\subset \mathrm{sh}(\Lambda, \delta)\subset \Lambda_{\varepsilon}$. Furthermore, we know that each $\delta$-pseudo orbit of $\Lambda$ is a $\delta$-pseudo orbit of $\tilde{\Lambda}$ so there exists a unique point $y\in\tilde{\Lambda}$ that is a $\varepsilon$-shadowing point of the pseudo orbit. Hence, $\mathrm{sh}(\Lambda, \delta)\subset \tilde{\Lambda}$. Let $\Lambda_1=\mathrm{sh}(\Lambda, \delta)$.

If $\Lambda_1=\Lambda$ we know from the previous proposition that $\Lambda$ is locally maximal. So suppose that $\Lambda_1\neq \Lambda$ and let $\nu_1=d_H(\Lambda_1, \Lambda)$ where $d_H(\cdot, \cdot)$ is the Hausdorff distance between the sets. More generally, let $\Lambda_{j+1}=\mathrm{sh}(\Lambda_j, \delta)$, $\nu_{j+1}=d_H(\Lambda_{j+1}, \Lambda_j)$ for $j\in\mathbb{N}$. By the shadowing estimates we know that $\nu_j\in(0, \varepsilon)$.

Fix $\gamma\in(0, \delta/4)$ such that for all $x,y\in M$ if $d(x,y)<\gamma$, then
$$d(f(x), f(y))<\delta/4\textrm{ and }d(f^{-1}(x), f^{-1}(y))<\delta/4.$$

 \begin{claim}
 For all $j\in\mathbb{N}$ if $\Lambda_j\neq \Lambda_{j+1}$ and $\Lambda_{j+1}\neq \Lambda_{j+2}$, then either $\nu_j\geq \gamma$ or $\nu_{j+1}\geq \gamma$.

 \end{claim}

\begin{proof}[Proof of Claim] Suppose that $\Lambda_j\neq \Lambda_{j+1}$ and $\Lambda_{j+1}\neq \Lambda_{j+2}$ and $\nu_j, \nu_{j+1}\in (0, \gamma)$. Fix $y\in \Lambda_{j+2}-\Lambda_{j+1}$. We will construct a $\delta$-pseudo orbit in $\Lambda_j$ that $y$ $\varepsilon$-shadows. This will show that $y\in \Lambda_{j+1}$, a contradiction.

 Then there exists some $y_0\in \Lambda_{j+1}$ and $x_0\in \Lambda_j$ such that $d(y, y_0)<\nu_{j+1}$ and $d(y_0, x_0)<\nu_{j}<\gamma$. Then
 $$d(f(y), f(y_0))<\delta/4\textrm{ and }d(f(y_0), f(x_0))<\delta/4.$$ Also, we know there exists points $y_1\in \Lambda_{j+1}$ and $x_1\in \Lambda_j$ such that $d(f(y), y_1)<\nu_{j+2}$ and $d(y_1, x_1)<\nu_{j+1}$. Hence,
 $$
 \begin{array}{llll}
 d(f(x_0), x_1) & \leq d(f(x_0), f(y_0)) + d(f(y_0), f(y)) + d(f(y), y_1) + d(y_1, x_1)\\
 &\leq \delta/4 +\delta/4 +\nu_{j+2} + \nu_{j+1}<\delta
 \end{array}
 $$
 and
 $$d(f(y), x_1)\leq d(f(y), y_1) + d(y_1, x_1)<\nu_{j+2} + \nu_{j+1}<\delta/2<\varepsilon.$$
 Continue inductively, we can construct a forward $\delta$-pseudo orbit $(x_k)_{k=0}^{\infty}$ such that $y$ $\varepsilon$-shadows the pseudo orbit. Also, since the estimates on $\gamma$ apply for $f^{-1}$ we can construct a bi-infinite $\delta$-pseudo orbit $(x_k)$ in $\Lambda_j$ such that $y$ $\varepsilon$ shadows $(x_k)$. Then $y\in\Lambda_{j+1}$, a contradiction.
\end{proof}

We now return to the proof of the theorem.

Repeating the above arguments for $\Lambda_j$ we see that $\Lambda_j\subset \tilde{\Lambda}$ for all $j\in\mathbb{N}$. We know from Proposition~\ref{p.stabilize} that if $\Lambda_{j+1}=\Lambda_j$ for some $j\in\mathbb{N}$ that $\Lambda_j$ is locally maximal. To conclude the proof of the theorem we simply need to show that the sequence $\Lambda_j$ stabilizes.

Suppose that the sequence $\Lambda_j$ does not stabilizer. We know from the above claim that for each $j$ the sets $\Lambda_{j+1}$ or $\Lambda_{j+2}$ will be a distance of $\gamma$ from the set $\Lambda_j$ using the Hausdorff metric. Since each manifold is finite dimensional we know that given the dimension of the manifold if the sequence does not stabilize there exists some $N\in\mathbb{N}$ where the Hausdorff distance from $\Lambda_N$ to $\Lambda$ is greater than $\eta$. This is a contradiction since each $\Lambda_j\subset \tilde{\Lambda}$ and $\tilde{\Lambda}\subset V_{\eta}(\Lambda)$.
\end{proof}

\end{document}